\documentclass{amsart}
\usepackage{amsmath,amsrefs,mathrsfs}
\usepackage{mathtools,enumerate}
\usepackage{tikz}
\usetikzlibrary{decorations.pathreplacing,calligraphy}

\usepackage[bookmarksnumbered, colorlinks, plainpages]{hyperref}
\hypersetup{colorlinks=true,linkcolor=red, anchorcolor=green, citecolor=cyan, urlcolor=blue, filecolor=magenta, pdftoolbar=true}
\usepackage[capitalise]{cleveref}

\usepackage{cancel}

\newtheorem{theorem}{Theorem}[section]
\newtheorem{lemma}[theorem]{Lemma}

\newtheorem{corollary}[theorem]{Corollary}
\newtheorem{problem}[theorem]{Problem}

\theoremstyle{definition}
\newtheorem{definition}[theorem]{Definition}

\theoremstyle{remark}
\newtheorem{remark}[theorem]{Remark}

\numberwithin{equation}{section}

\begin{document}
\title[$\mathcal{C}^c-\mathcal{C}$]{The algebraic difference of a Cantor set and its complement}

\author[P. Nowakowski]{Piotr Nowakowski}
\address{Faculty of Mathematics and Computer Science,\newline \indent University of {\L}\'{o}d\'{z}, Banacha 22, 90-238 {\L}\'{o}d\'{z}, Poland\newline \indent ORCID: 0000-0002-3655-4991}
\email{\href{mailto:piotr.nowakowski@wmii.uni.lodz.pl}{piotr.nowakowski@wmii.uni.lodz.pl}}

\author[C.-H. Pan]{Cheng-Han Pan}
\address{Department of Mathematics and Statistics,\newline \indent Mount Holyoke College, South Hadley,\newline \indent Massachusetts 01075-1461, United States}
\email{\href{mailto:cpan@mtholyoke.edu}{cpan@mtholyoke.edu}}



\date{Draft of July 24, 2025}
\dedicatory{}

\commby{}

\begin{abstract}
Let $\mathcal{C}\subseteq[0,1]$ be a Cantor set. In the classical $\mathcal{C}\pm\mathcal{C}$ problems, modifying the ``size'' of $\mathcal{C}$ has a magnified effect on $\mathcal{C}\pm\mathcal{C}$. However, any gain in $\mathcal{C}$ necessarily results in a loss in $\mathcal{C}^c$, and vice versa. This interplay between $\mathcal{C}$ and its complement $\mathcal{C}^c$ raises interesting questions about the delicate balance between the two, particularly in how it influences the ``size'' of $\mathcal{C}^c-\mathcal{C}$. One of our main results indicates that the Lebesgue measure of $\mathcal{C}^c-\mathcal{C}$ has a greatest lower bound of $\frac{3}{2}$.
\end{abstract}

\maketitle

\section{Introduction}
Let $\mathfrak{C}\subseteq[0,1]$ denote the classical Cantor ternary set. A standard construction of $\mathfrak{C}$ is to iteratively remove the open middle third of each interval in the current set, starting with the interval $[0,1]$. Despite the fact that $\mathfrak{C}$ is nowhere dense and has zero Lebesgue measure, it is well-known that the algebraic difference $\mathfrak{C}-\mathfrak{C}=\{x-y\in\mathbb{R}\colon x,y\in\mathfrak{C}\}$ is exactly the closed interval $[-1,1]$ (see \cite{St}). Another beautiful proof of this result can be found in \cite[sec. 3, ch, 8]{MR1996162}. Of course, the algebraic sum and difference of a vast variation of Cantor sets has been studied extensively in several papers (e.g. \cites{MR2534296,MR1289273,MR1637408,MR1267692,MR2566156,MR4535243,MR4824840,MR1491873,MR4574153,MR1153749,MR1617830,MR3950055}). This topic found applications for example in number theory (see \cite{MR22568}), dynamical systems (see e.g. \cite{MR893866}) or spectral theory (see \cite{MR2802305}). We asked ourselves a question what would happen if we replace one Cantor set in the difference by its complement. In this paper, our primary focus is on understanding the ``size'' of such a hybrid difference set $\mathcal{C}^c-\mathcal{C}$, where the relevant notions and terminology are introduced below.

\begin{definition}\label{DC}
A Cantor set $\mathcal{C}\subseteq[a,b]$ is a nowhere dense, perfect subset of $[a,b]$ that contains both endpoints $a$ and $b$. We denote its complement in $[a,b]$ by $\mathcal{C}^c$.
\end{definition}

Unless otherwise specified, we work with Cantor sets on $[0,1]$. To motivate the discussion, we start with the following question.

\begin{problem}\label{PC}
Is it true that $\mathfrak{C}^c-\mathfrak{C}=[-1,1]$? If not, how does it look like?
\end{problem}

One may notice that $-1,0,1\not\in\mathfrak{C}^c-\mathfrak{C}$ in a quick observation. In particular, $-1$ can only be written as $0-1$, but $0\not\in\mathfrak{C}^c$. $1$ can only be written as $1-0$, but $1\not\in\mathfrak{C}^c$. $0$ can only be written as $x-x$, but $x\in\mathfrak{C}^c$ and $x\in\mathfrak{C}$ cannot happen simultaneously. Does it miss any more values in $[-1,1]$? Yes, $[-1,1]\setminus(\mathfrak{C}^c-\mathfrak{C})$ is in fact countably infinite, and we will identify specifically each value which $\mathfrak{C}^c-\mathfrak{C}$ misses in $[-1,1]$ in \cref{CCP}. This naturally raises several questions about the "size" of the set $\mathcal{C}^c-\mathcal{C}$ for a general Cantor set $\mathcal{C}\subseteq[0,1]$. Our findings are listed below:
\begin{itemize}
\item[•] Some $\mathcal{C}^c-\mathcal{C}$ misses only $-1$, $0$, $1$ from $[-1,1]$. See \cref{TS3}.
\item[•] Some $\mathcal{C}^c-\mathcal{C}$ always misses a countable set from $[-1,1]$. See \cref{CCI}.
\item[•] Some $\mathcal{C}^c-\mathcal{C}$ misses a ``fat'' Cantor set from $[-1,1]$. See \cref{CSFC}.
\item[•] The Lebesgue measure of $\mathcal{C}^c-\mathcal{C}$ has a greatest lower bound of $\frac{3}{2}$. See \cref{CSPM}.
\end{itemize}

\section{Notations and Two Elementary Lemmas}
We begin by stating some general notations and two general lemmas that serve our future arguments. In particular, \cref{LLG} describes what $\mathcal{C}^c-\mathcal{C}$ must contain, and \cref{LS} describes what $\mathcal{C}^c-\mathcal{C}$ must not contain.

\begin{definition}\label{DG}\:
\begin{enumerate}[(i)]
\item A gap $G$ of a Cantor set $\mathcal{C}\subseteq[a,b]$ refers to a connected component of $\mathcal{C}^c$.
\item Let $I\subseteq\mathbb{R}$ be an interval, we denote $l(I)$ as its left end point, $r(I)$ as its right endpoint, $c(I)$ as its middle point, and $|I|$ as its length.
\end{enumerate}
\end{definition}

\begin{lemma}\label{LLG}
Let $\mathcal{C}\subseteq[0,1]$ be a Cantor set, and let $G$ be a gap of $\mathcal{C}$. Let $a\leq b$ be points in $\mathcal{C}$ such that $[a,b]$ does not contain $G$.
\begin{enumerate}[(i)]
\item If $G$ is strictly longer than every gap of $\mathcal{C}\cap[a,b]$,
\begin{itemize}
\item[] then $(l(G)-b,r(G)-a)\subseteq\mathcal{C}^c-\mathcal{C}$.
\end{itemize}
\item If $G$ is longer than or equal to every gap of $\mathcal{C}\cap[a,b]$,
\begin{itemize}
\item[] then there is a finite set $F$ such that $(l(G)-b,r(G)-a)\setminus F\subseteq\mathcal{C}^c-\mathcal{C}$.
\end{itemize}
\end{enumerate}
\end{lemma}
\begin{proof}
Since $G\subseteq\mathcal{C}^c$, it is easy to see that $(G-y)\cap\mathcal{C}\neq\emptyset$ implies $y\in\mathcal{C}^c-\mathcal{C}$. Indeed, if $(G-y)\cap\mathcal{C}\neq\emptyset$, then there are $g\in G\subseteq\mathcal{C}^c$ and $c\in\mathcal{C}$ such that $g-y=c$, or equivalently, $y=g-c\in\mathcal{C}^c-\mathcal{C}$. To prove (i), it suffices to show that
\begin{center}
$(G-y)\cap\mathcal{C}\neq\emptyset$ for every $y\in(l(G)-b,r(G)-a)$.
\end{center}
Let $y\in(l(G)-b,r(G)-a)$. Then
\begin{equation*}
G-y\subseteq(l(G)-(r(G)-a),r(G)-(l(G)-b))=(a-|G|,b+|G|)\mbox{,}
\end{equation*}
and clearly $|G-y|=|G|$. If $G-y\subseteq[a,b]$, then since all gaps of $[a,b]$ are shorter than $G$, then $G-y$ must intersect $\mathcal{C}$, that is, there is $c\in(G-y)\cap\mathcal{C}$. On the other hand, if $G-y\not\subseteq[a,b]$, then $G-y$ must intersect the boundary of $[a,b]$, that is, either $a\in G-y$ or $b\in G-y$. Since both $a$ and $b$ belong also to $\mathcal{C}$, it again follows that $(G-y)\cap\mathcal{C}\neq\emptyset$. See \cref{FG-y}.
\begin{figure}
\begin{tikzpicture}[scale=10]
\draw[](0,0)--(1,0)
node[pos=0,anchor=east]{$[0,1]=$}
node[pos=0,anchor=north]{$0$}
node[pos=8/10,anchor=south]{$G$}
node[pos=1,anchor=north]{$1$}
node[pos=1,anchor=south west]{$\mathcal{C}$};
\draw[line width=4pt,dashed](0/10,0)--(0.5/10,0);
\draw[line width=4pt,dashed](1/10,0)--(1.5/10,0);
\draw[line width=4pt,dashed](2/10,0)--(2.5/10,0);
\draw[line width=4pt,dashed](3/10,0)--(4/10,0);
\draw[line width=4pt,dashed](4.4/10,0)--(4.6/10,0);
\draw[line width=4pt,dashed](5/10,0)--(7/10,0);
\draw[line width=4pt,dashed](9/10,0)--(10/10,0);

\draw[](7/10,0.025)--(7/10,-0.025)node[pos=1,anchor=north]{$l(G)$};
\draw[](9/10,0.025)--(9/10,-0.025)node[pos=1,anchor=north]{$r(G)$};

\node[scale=1.5]at(7/10+0.005,0){$($};
\node[scale=1.5]at(9/10-0.005,0){$)$};
\node[scale=2]at(3/10,0){$\textbf{[}$};
\node[scale=2]at(6/10,0){$\textbf{]}$};

\node[scale=1.5]at(1/10+0.005,0.075){$\textbf{(}$};
\node[scale=1.5]at(3/10-0.005,0.075){$)$};
\node[scale=0.8]at(2/10,0.125){$G-(r(G)-a)$};
\node[scale=0.8]at(2/10,0.09){Does not};
\node[scale=0.8]at(2/10,0.06){include $a$};

\node[scale=1.5]at(3.5/10+0.005,0.075){$($};
\node[scale=1.5]at(5.5/10-0.005,0.075){$)$};
\node[]at(4.5/10,0.125){$G-y$};
\node[scale=0.8]at(4.5/10,0.09){Longer than};
\node[scale=0.7]at(4.5/10,0.06){any gap of $[a,b]$};

\node[scale=1.5]at(6/10+0.005,0.075){$($};
\node[scale=1.5]at(8/10-0.005,0.075){$\textbf{)}$};
\node[scale=0.8]at(7/10,0.125){$G-(l(G)-b)$};
\node[scale=0.8]at(7/10,0.09){Does not};
\node[scale=0.8]at(7/10,0.06){include $b$};

\draw[](3/10,0.1)--(3/10,-0.03)node[pos=1,anchor=north]{$a$};
\draw[](6/10,0.1)--(6/10,-0.03)node[pos=1,anchor=north]{$b$};
\draw[<-](3.9/10,0.03)--(3.9/10,-0.03)node[pos=1,anchor=north]{$c$};

\draw[thick](1/10,0.1)--(1/10,-0.05)node[pos=1,anchor=north]{$a-|G|$};
\draw[thick](8/10,0.1)--(8/10,0.045);
\draw[thick](8/10,0.005)--(8/10,-0.05)node[pos=1,anchor=north]{$b+|G|$};
\node[xscale=-1,yscale=1]at(3.25/10,0.075){$\rightsquigarrow$};
\node[]at(5.75/10,0.075){$\rightsquigarrow$};
\end{tikzpicture}
\caption{Illustration of the key idea in the proof of (i) of \cref{LLG}. To ensure that $G-y$ intersects $\mathcal{C}\cap[a,b]$, $G-y$ must be restricted within $(a-|G|,b+|G|)$. This requires that the value of $y$ lies strictly between $l(G)-b$ and $r(G)-a$. Depending on the position of $G-y$, it intersects either some $c\in\mathcal{C}\cap[a,b]$, or one of the endpoints $a\in\mathcal{C}$ or $b\in\mathcal{C}$.}\label{FG-y}
\end{figure}
Therefore, $(G-y)\cap\mathcal{C}\neq\emptyset$ for every $y\in(l(G)-b,r(G)-a)$.

The proof of (ii) is analogous except for those $y \in (l(G)-b,r(G)-a)$ for which $G-y=H'$ for some gap $H'$ contained in $[a,b]$. See \cref{FGHH}.
\begin{figure}
\begin{tikzpicture}[scale=10]
\draw[](0,0)--(1,0)
node[pos=0,anchor=east]{$[0,1]=$}
node[pos=0,anchor=north]{$0$}
node[pos=2/10,anchor=south]{$G-y'$}
node[pos=2/10,anchor=north]{$H'$}
node[pos=4.5/10,anchor=south]{$G-y''$}
node[pos=4.5/10,anchor=north]{$H''$}
node[pos=8/10,anchor=south]{$G$}
node[pos=1,anchor=north]{$1$}
node[pos=1,anchor=south west]{$\mathcal{C}$};
\draw[line width=4pt,dashed](0/10,0)--(1/10,0);
\draw[line width=4pt,dashed](3/10,0)--(3.5/10,0);
\draw[line width=4pt,dashed](5.5/10,0)--(7/10,0);
\draw[line width=4pt,dashed](9/10,0)--(10/10,0);

\draw[](1/10,0.025)--(1/10,-0.025)node[pos=1,anchor=north]{$l(H')$};
\draw[](3.5/10,0.025)--(3.5/10,-0.025)node[pos=1,anchor=north]{$l(H'')$};
\draw[](7/10,0.025)--(7/10,-0.025)node[pos=1,anchor=north]{$l(G)$};
\draw[](9/10,0.025)--(9/10,-0.025)node[pos=1,anchor=north]{$r(G)$};

\node[scale=1.5]at(7/10+0.005,0){$($};
\node[scale=1.5]at(9/10-0.005,0){$)$};
\node[scale=2]at(0.3/10,0){$\textbf{[}$};
\node[scale=2]at(6/10,0){$\textbf{]}$};

\node[scale=1.5]at(1/10+0.005,0){$($};
\node[scale=1.5]at(3/10-0.005,0){$)$};

\node[scale=1.5]at(3.5/10+0.005,0){$($};
\node[scale=1.5]at(5.5/10-0.005,0){$)$};

\draw[](0.3/10,0)--(0.3/10,-0.03)node[pos=1,anchor=north]{$a$};
\draw[](6/10,0)--(6/10,-0.03)node[pos=1,anchor=north]{$b$};
\end{tikzpicture}
\caption{Illustration of the key idea in the proof of (ii) of \cref{LLG}. $H'$ and $H''$ are gaps in $\mathcal{C}\cap[a,b]$ with the same length as $G$. $G$ can be completely shifted into $H'$ and $H''$ by some $y'=l(G)-l(H')$ and $y''=l(G)-l(H'')$ respectively. Besides $H'$ and $H''$, $G-y$ must still intersect $\mathcal{C}\cap[a,b]$ at somewhere since no other gap is long enough to contain $G-y$.
}\label{FGHH}
\end{figure}
But this situation can occur only finitely many times since $[a,b]$ cannot host any infinite number of gaps of the same length. Therefore, there exists a finite set $F$ such that
\begin{equation*}
(l(G)-b,r(G)-a)\setminus F\subseteq\mathcal{C}^c-\mathcal{C}\mbox{.}
\end{equation*}
\end{proof}

\begin{lemma}\label{LS}
Let $\mathcal{C}\subseteq[0,1]$ be a Cantor set, and define $S\coloneqq[-1,1]\setminus(\mathcal{C}^c-\mathcal{C})$. For any nonempty $Y\subseteq[-1,1]$,
\begin{center}
\begin{tabular}{rcl}
$Y\subseteq S$&if and only if&$(\mathcal{C}+Y)\cap[0,1]\subseteq\mathcal{C}$.
\end{tabular}
\end{center}
\end{lemma}
\begin{proof}
Let $Y\subseteq[-1,1]$. Suppose there are $c\in\mathcal{C}$ and $y\in Y$ such that $c+y\in[0,1]$, but $c+y\not\in\mathcal{C}$. Then there is $x\in\mathcal{C}^c$ such that $x=c+y$, and we can write $y=x-c\in\mathcal{C}^c-\mathcal{C}$, and so $Y\not\subseteq S$.

Conversely, suppose that $(\mathcal{C}+Y)\cap[0,1]\subseteq C$ and $Y\not\subseteq S$. Take $y\in Y\setminus S$. Then by the definition of $S$, there must exist $x\in\mathcal{C}^c$ and $c\in\mathcal{C}$ such that $y=x-c$. This implies $c+y=x\in\mathcal{C}+Y$, but since $x\in\mathcal{C}^c=[0,1]\setminus\mathcal{C}$, we have $(C+Y)\cap[0,1]\not\subseteq\mathcal{C}$, contradicting the assumption.
\end{proof}

\section{Case of central Cantor sets}\label{S2}
Recall that the classical Cantor ternary set $\mathfrak{C}\subseteq[0,1]$ can be constructed by iteratively removing the open middle third of each interval at every stage, starting with the interval $[0,1]$. An immediate generalization of this process is to remove the open middle portion of relative length $a_n\in(0,1)$ from each interval at the $n$th step. Following the notation and definitions in \cite{MR4535243}, let $\textbf{a}\coloneqq(a_n)\in(0,1)^{\mathbb{N}}$ be a sequence, and its corresponding central Cantor set $\mathcal{C}(\textbf{a})\subseteq[0,1]$ is then constructed as illustrated in \cref{FCC}.
\begin{figure}
\begin{tikzpicture}[scale=10]
\draw[](0,0+0.4)--(1,0+0.4)
node[pos=0,anchor=east]{$I=[0,1]=$};
\draw[line width=4pt](0,0+0.4)--(1/3,0+0.4);
\draw[line width=4pt](2/3,0+0.4)--(1,0+0.4);
\node[anchor=south]at(1/6,0.02+0.4){$I_{0}$};
\node[anchor=south]at(1/2,0+0.4){$P$};
\node[anchor=south]at(5/6,0.02+0.4){$I_{1}$};
\draw[decorate,decoration=brace,thick](1/3-0.01,0-0.02+0.4)--(0+0.01,0-0.02+0.4)node[pos=1/2,anchor=north]{$\frac{1-a_1}{2}$};
\draw[decorate,decoration=brace,thick](2/3-0.01,0-0.02+0.4)--(1/3+0.01,0-0.02+0.4)node[pos=1/2,anchor=north]{$a_{1}$};
\draw[decorate,decoration=brace,thick](1-0.01,0-0.02+0.4)--(2/3+0.01,0-0.02+0.4)node[pos=1/2,anchor=north]{$\frac{1-a_1}{2}$};

\draw[](0,0+0.2)--(1,0+0.2)
node[pos=0,anchor=east]{$I_0=$};
\draw[line width=4pt](0,0+0.2)--(3/8,0+0.2);
\draw[line width=4pt](5/8,0+0.2)--(1,0+0.2);
\node[anchor=south]at(1/6,0.02+0.2){$I_{00}$};
\node[anchor=south]at(1/2,0+0.2){$P_0$};
\node[anchor=south]at(5/6,0.02+0.2){$I_{01}$};
\draw[decorate,decoration=brace,thick](3/8-0.01,0-0.02+0.2)--(0+0.01,0-0.02+0.2)node[pos=1/2,anchor=north]{$\frac{1-a_1}{2}\cdot\frac{1-a_2}{2}$};
\draw[decorate,decoration=brace,thick](5/8-0.01,0-0.02+0.2)--(3/8+0.01,0-0.02+0.2)node[pos=1/2,anchor=north]{$ a_{2}\cdot\frac{1-a_1}{2}$};
\draw[decorate,decoration=brace,thick](1-0.01,0-0.02+0.2)--(5/8+0.01,0-0.02+0.2)node[pos=1/2,anchor=north]{$\frac{1-a_1}{2}\cdot\frac{1-a_2}{2}$};

\draw[dashed,->,line width=1pt](0,0-0.02+0.4)--(0,-0.2+0.02+0.4);
\draw [dashed,->,line width=1pt](1/3,0.4-0.02)to[out=-85,in=175](5/9,0.3-0.01)to[out=-5,in=175](7/9,0.3-0.02)to[out=-5,in=95](1-0.01,0.2+0.02);

\draw[](0,0)--(1,0)
node[pos=0,anchor=east]{$I_{t_1t_2\ldots t_n}=$};
\draw[line width=4pt](0,0)--(1/4,0);
\draw[line width=4pt](3/4,0)--(1,0);
\node[anchor=south]at(1/8,0.02){$I_{t_1t_2\ldots t_n0}$};
\node[anchor=south]at(1/2,0){$P_{t_1t_2\ldots t_n}$};
\node[anchor=south]at(7/8,0.02){$I_{t_1t_2\ldots t_n1}$};
\draw[decorate,decoration=brace,thick](1/4-0.01,0-0.02)--(0+0.01,0-0.02)node[pos=1/2,anchor=north]{$\prod_{k=1}^{n+1}\frac{1-a_k}{2}$};
\draw[decorate,decoration=brace,thick](3/4-0.01,0-0.02)--(1/4+0.01,0-0.02)node[pos=1/2,anchor=north]{$a_{n+1}\prod_{k=1}^{n}\frac{1-a_k}{2}$};
\draw[decorate,decoration=brace,thick](1-0.01,0-0.02)--(3/4+0.01,0-0.02)node[pos=1/2,anchor=north]{$\prod_{k=1}^{n+1}\frac{1-a_k}{2}$};
\end{tikzpicture}
\caption{Let $\textbf{a}\in(0,1)^{\mathbb{N}}$. The construction of a central Cantor set $\mathcal{C}(\textbf{a})\subseteq[0,1]$ starts with removing $P=(\frac{1-a_1}{2},\frac{1+a_1}{2})$, the open middle $a_1$ portion of $[0,1]$, from $[0,1]$. The remaining two intervals are denoted by $I_0$ on the left and $I_1$ on the right. The second iteration is then applied on both $I_0$ and $I_1$. In particular, removing $P_0$, the middle $a_2$ portion of $I_0$, from $I_0$ yields $I_{00}$ and $I_{01}$, and removing $P_1$, the middle $a_2$ portion of $I_1$, from $I_1$ yields $I_{10}$ and $I_{11}$. As the iteration goes on, $I_{t_1t_2\ldots t_n}$ represents a remaining subinterval at the end of $n$th step, where $t_1t_2\ldots t_n$ is a binary sequence of length $n$, and $\mathcal{C}(\textbf{a})\coloneqq\bigcap_{n=1}^{\infty}\bigcup_{\textbf{t}\in\{0,1\}^n}I_\textbf{t}$.}\label{FCC}
\end{figure}

Observe that for any $n\in\mathbb{N}$ the procedure of removing gaps in all intervals $I_\textbf{t}$ for $\textbf{t}\in\{0,1\}^n$ is identical. Therefore, we have the following fact.

\begin{lemma}\label{Lbrick}
Let $\textbf{\em a}\in(0,1)^\mathbb{N}$. For any $n\in\mathbb{N}$ and any $\textbf{\em t}\in\{0,1\}^n$, the sets $\mathcal{C}(\textbf{a})\cap I_{\underbrace{\scriptstyle00\ldots0}_{n}}$ and $\mathcal{C}(\textbf{\em a})\cap I_\textbf{\em t}$ are identical up to a shift.
\end{lemma}

In this section, we consider the class of central Cantor sets $\mathcal{C}(\textbf{a})\subseteq[0,1]$ and show that $\mathcal{C}(\textbf{a})^c-\mathcal{C}(\textbf{a})$ would always miss a countably infinite subset from $[-1,1]$.

\begin{theorem}\label{TALC}
For every $\textbf{\em a}\in(0,1)^{\mathbb{N}}$, the set $S\coloneqq[-1,1]\setminus(\mathcal{C}(\textbf{\em a})^c-\mathcal{C}(\textbf{\em a}))$ is at least countably infinite. In particular,
\begin{equation*}
S\supseteq\{0,\pm r(P),\pm r(P_1),\pm r(P_{11}),\ldots,\pm1\}\mbox{.}
\end{equation*}
\end{theorem}
\begin{proof}
It is trivial that $S$ always contains $\{0,\pm1\}$. By \cref{Lbrick}, $\mathcal{C}(\textbf{a})\cap I_{\underbrace{{\scriptstyle00\ldots0}}_{n}}$ and $\mathcal{C}(\textbf{a})\cap I_{\underbrace{{\scriptstyle11\ldots1}}_{n}}$ are identical up to a shift. In particular, \begin{equation*}
(\mathcal{C}(\textbf{a})\cap I_{\underbrace{\scriptstyle00\ldots0}_{n}})+l(I_{\underbrace{\scriptstyle11\ldots1}_{n}})=\mathcal{C}(\textbf{a})\cap I_{\underbrace{\scriptstyle11\ldots1}_{n}}\mbox{,}
\end{equation*}
and equivalently
\begin{equation*}
(\mathcal{C}(\textbf{a})\cap I_{\underbrace{\scriptstyle11\ldots1}_{n}})-l(I_{\underbrace{\scriptstyle11\ldots1}_{n}})=\mathcal{C}(\textbf{a})\cap I_{\underbrace{\scriptstyle00\ldots0}_{n}}\mbox{.}
\end{equation*}
Since $\mathcal{C}(\textbf{a})\cap I_{\underbrace{\scriptstyle00\ldots0}_{n}}$ and $\mathcal{C}(\textbf{a})\cap I_{\underbrace{\scriptstyle11\ldots1}_{n}}$ are located at the two far ends of $\mathcal{C}(\textbf{a})$, we can interpret this as $(\mathcal{C}(\textbf{a})\pm l(I_{\underbrace{\scriptstyle11\ldots1}_{n}}))\cap[0,1]\subseteq\mathcal{C}(\textbf{a})$. See \cref{FC2FE}.
\begin{figure}
\begin{tikzpicture}[scale=10]
\draw[](0,0)--(1,0)
node[pos=0,anchor=east]{$[0,1]=$}
node[pos=2/10,anchor=south]{$P_0$}
node[pos=5/10,anchor=south]{$P$}
node[pos=8/10,anchor=south]{$P_{1}$}
node[pos=1,anchor=south west]{$\mathcal{C}(\textbf{a})$};

\node[scale=2]at(0/10,0){$\textbf{[}$};
\node[scale=2]at(1.5/10,0){$\textbf{]}$};
\node[scale=2]at(8.5/10,0){$\textbf{[}$};
\node[scale=2]at(10/10,0){$\textbf{]}$};

\node[anchor=north]at(0/10,-0.03){$0$};
\node[anchor=north]at(10/10,-0.03){$1$};

\draw[line width=4pt,dashed,->](0/10,0)--(1.5/10,0);
\draw[line width=4pt,dashed](2.5/10,0)--(4/10,0);
\draw[line width=4pt,dashed](6/10,0)--(7.5/10,0);
\draw[line width=4pt,dashed,->](8.5/10,0)--(10/10,0);
\node[anchor=south]at(0.75/10,0.02){$I_{00}$};
\node[anchor=south]at(3.25/10,0.02){$I_{01}$};
\node[anchor=south]at(6.75/10,0.02){$I_{10}$};
\node[anchor=south]at(9.25/10,0.02){$I_{11}$};

\draw[decorate,decoration=brace,thick](0/10,0+0.08)--(8.5/10,0+0.08)node[pos=1/2,anchor=south]{$l(I_{11})=r(P_1)$};
\draw[->](0/10,0+0.08)--(0/10,0+0.04);
\draw[->](8.5/10,0+0.08)--(8.5/10,0+0.04);
\end{tikzpicture}
\caption{Illustration of the key idea in the proof of \cref{TALC}. Since the two far ends, $\mathcal{C}(\textbf{a})\cap I_{00}$ and $\mathcal{C}(\textbf{a})\cap I_{11}$, are identical upto a shift, they can be shifted into each other by $\pm r(P_{1})$. This operation can be applied to shift the entire set and then trim it back within the interval $[0,1]$. In particular, $(\mathcal{C}(\textbf{a})\pm r(P_{1}))\cap[0,1]\subseteq\mathcal{C}(\textbf{a})$.}\label{FC2FE}
\end{figure}

Notice $l(I_{\underbrace{\scriptstyle11\ldots1}_{n}})=r(P_{\underbrace{\scriptstyle1\ldots1}_{n-1}})$ and let $Y\coloneqq\{0,\pm r(P),\pm r(P_1),\pm r(P_{11}),\ldots,\pm1\}$. Clearly, $(\mathcal{C}(\textbf{a})+Y)\cap[0,1]\subseteq\mathcal{C}(\textbf{a})$. By \cref{LS}, we have $Y\subseteq S$.
\end{proof}

\begin{theorem}\label{T13}
For every $\textbf{\em a}\in[\frac{1}{3},1)^{\mathbb{N}}$, the set $S\coloneqq[-1,1]\setminus(\mathcal{C}(\textbf{\em a})^c-\mathcal{C}(\textbf{\em a}))$ is fully determined. In particular,
\begin{equation*}
S=\{0,\pm r(P),\pm r(P_1),\pm r(P_{11}),\ldots,\pm1\}\mbox{.}
\end{equation*}
\end{theorem}
\begin{proof}
By \cref{TALC}, we already have $S\supseteq\{0,\pm r(P),\pm r(P_1),\pm r(P_{11}),\ldots,\pm1\}$. To show that they are equal, it suffices to prove that $\mathcal{C}(\textbf{a})^c-\mathcal{C}(\textbf{a})$ contains all the following open intervals,
\begin{equation*}
\ldots\mbox{, }(-r(P_1),-r(P))\mbox{, }(-r(P),0)\mbox{, }(0,r(P))\mbox{, }(r(P),r(P_1))\mbox{,}\ldots\mbox{,}
\end{equation*}
which cover all the gaps within $\{-1,\ldots,-r(P_1),-r(P),0,r(P),r(P_1),\ldots,1\}$.

Indeed, if $\textbf{a}\subseteq[\frac{1}{3},1)^\mathbb{N}$, the assumption implies that
\begin{enumerate}[]
\item $P_{\underbrace{\scriptstyle11\ldots1}_{n}}$ is strictly longer than every gap of $\mathcal{C}(\textbf{a})\cap I_{\underbrace{\scriptstyle00\ldots0}_{n+1}}=\mathcal{C}(\textbf{a})\cap[0,|I_{\underbrace{\scriptstyle00\ldots0}_{n+1}}|]$, and
\item $P_{\underbrace{\scriptstyle00\ldots0}_{n}}$ is strictly longer than every gap of $\mathcal{C}(\textbf{a})\cap I_{\underbrace{\scriptstyle11\ldots1}_{n+1}}=\mathcal{C}(\textbf{a})\cap[1-|I_{\underbrace{\scriptstyle11\ldots1}_{n+1}}|,1]$.
\end{enumerate}
By (i) of \cref{LLG} and a visual assist in \cref{FP-I},
\begin{figure}
\begin{tikzpicture}[scale=10]
\draw[](0,0)--(1,0)
node[pos=0,anchor=east]{$[0,1]=$}
node[pos=0,anchor=north]{$0$}
node[pos=1.6/10,anchor=south]{$P_{\underbrace{\scriptstyle00\ldots0}_{n}}$}
node[pos=6/10,anchor=south]{$P_{\underbrace{\scriptstyle1\ldots1}_{n-1}}$}
node[pos=8.4/10,anchor=south]{$P_{\underbrace{\scriptstyle11\ldots1}_{n}}$}
node[pos=1,anchor=north]{$1$}
node[pos=1,anchor=south west]{$\mathcal{C}(\textbf{a})$};
\draw[line width=4pt,dashed](0/10,0)--(1/10,0);
\draw[line width=4pt,dashed](2.2/10,0)--(3.2/10,0);
\draw[line width=4pt,dashed](6.8/10,0)--(7.8/10,0);
\draw[line width=4pt,dashed](9/10,0)--(10/10,0);
\node[anchor=south]at(0.5/10,0.02){$I_{\underbrace{\scriptstyle00\ldots00}_{n+1}}$};
\node[anchor=south]at(7.3/10,0.02){$I_{\underbrace{\scriptstyle11\ldots10}_{n+1}}$};

\filldraw[white](0.46,-0.01)rectangle(0.535,0.01);
\node[]at(1/2,0){$\cdots$};
\node[rotate=90,scale=1.5]at(0.46,0){$\sim$};
\node[rotate=90,scale=1.5]at(0.54,0){$\sim$};

\draw[](1/10,-0.03)--(1/10,+0.03);
\draw[](6.8/10,-0.03)--(6.8/10,+0.03);
\draw[](7.8/10,-0.03)--(7.8/10,+0.03);

\draw[line width=1pt,->](1.3/10,-0.05)--(1/10+0.005,-0.005);
\draw[line width=1pt,->](6.5/10,-0.05)--(6.8/10-0.005,-0.005);
\draw[line width=1pt,->](8.1/10,-0.05)--(7.8/10+0.005,-0.005);

\node[anchor=north]at(1.6/10,-0.05){$l(P_{\underbrace{\scriptstyle00\ldots0}_{n}})$};
\node[anchor=north]at(6/10,-0.05){$r(P_{\underbrace{\scriptstyle1\ldots1}_{n-1}})$};
\node[anchor=north]at(8.4/10,-0.05){$l(P_{\underbrace{\scriptstyle11\ldots1}_{n}})$};
\end{tikzpicture}
\caption{Illustration of a computation in the proof of \cref{T13}. Note that $I_{00\ldots00}$ and $I_{11\ldots10}$ have the same length. Therefore, $|I_{\protect\underbrace{\scriptstyle00\ldots00}_{n+1}}|=|I_{\protect\underbrace{\scriptstyle11\ldots10}_{n+1}}|=
l(P_{\protect\underbrace{\scriptstyle11\ldots1}_{n}})-r(P_{\protect\underbrace{\scriptstyle1\ldots1}_{n-1}})$.}\label{FP-I}
\end{figure}
we have
\begin{align*}
(l(P_{\underbrace{\scriptstyle11\ldots1}_{n}})-|I_{\underbrace{\scriptstyle00\ldots00}_{n+1}}|,r(P_{\underbrace{\scriptstyle11\ldots1}_{n}})-0)&=(l(P_{\underbrace{\scriptstyle11\ldots1}_{n}})-|I_{\underbrace{\scriptstyle11\ldots10}_{n+1}}|,r(P_{\underbrace{\scriptstyle11\ldots1}_{n}}))\\
&=(l(P_{\underbrace{\scriptstyle11\ldots1}_{n}})-
(l(P_{\underbrace{\scriptstyle11\ldots1}_{n}})-r(P_{\underbrace{\scriptstyle1\ldots1}_{n-1}})),r(P_{\underbrace{\scriptstyle11\ldots1}_{n}}))\\
&=(r(P_{\underbrace{\scriptstyle1\ldots1}_{n-1}}),r(P_{\underbrace{\scriptstyle11\ldots1}_{n}}))\subseteq\mathcal{C}(\textbf{a})^c-\mathcal{C}(\textbf{a})\mbox{.}
\end{align*}
Symmetrically, $(-r(P_{\underbrace{\scriptstyle11\ldots1}_{n}}),-r(P_{\underbrace{\scriptstyle11\ldots1}_{n-1}}))\subseteq\mathcal{C}(\textbf{a})^c-\mathcal{C}(\textbf{a})$ can be obtained in the same way. Therefore, we conclude that $S=\{0,\pm r(P),\pm r(P_1),\pm r(P_{11}),\ldots,\pm1\}$.
\end{proof}

Notice that the classical Cantor ternary set $\mathfrak{C}\subseteq[0,1]$ is actually a central Cantor set $\mathcal{C}(\textbf{a})$, where $\textbf{a}$ is a constant sequence of $\frac{1}{3}$. The next corollary provides a full answer to \cref{PC}.

\begin{corollary}\label{CCP}
Let $\mathfrak{C}\subseteq[0,1]$ denote the classical Cantor ternary set. Then
\begin{equation*}
\textstyle[-1,1]\setminus(\mathfrak{C}^c-\mathfrak{C})=\{0,\pm\frac{2}{3},\pm\frac{8}{9},\pm\frac{26}{27},\ldots,\pm1\}\mbox{.}
\end{equation*}
\end{corollary}

By \cref{TALC}, we know that $[-1,1]\setminus(\mathcal{C}(\textbf{a})^c-\mathcal{C}(\textbf{a}))$ is at least countably infinite. In the next theorem, we show that it is also at most countably infinite.

\begin{theorem}\label{TAMC}
For every $\textbf{\em a}\in(0,1)^\mathbb{N}$, the set $S\coloneqq[-1,1]\setminus(\mathcal{C}(\textbf{\em a})^c-\mathcal{C}(\textbf{\em a}))$ is at most countably infinite.
\end{theorem}
\begin{proof}
We will only show that $S\cap[0,1]$ is at most countably infinite. The argument for $S\cap[-1,0]$ follows symmetrically.

Let $G_1$ be the rightmost longest gap of $\mathcal{C}(\textbf{a})$. Since $G_1$ is longer than or equal to every gap of $\mathcal{C}(\textbf{a})\cap[0,l(G_1)]$, we have
\begin{center}
$(l(G_1)-l(G_1),r(G_1)-0)=(0,r(G_1))$
\end{center}
and by (ii) of \cref{LLG}, $(0,r(G_1)) \setminus F_1\subseteq\mathcal{C}(\textbf{a})^c-\mathcal{C}(\textbf{a})$ for some finite set $F_1$.

Now, let $G_2$ be the rightmost longest gap of $\mathcal{C}(\textbf{a})\cap[r(G_1),1]$. Observe that the interval $[r(G_1),1]$ is equal to an interval $I_{\underbrace{\scriptstyle11\ldots1}_{k}}$ for some $k\in \mathbb{N}$, and so, by \cref{Lbrick}, the sets $\mathcal{C}(\textbf{a})\cap[r(G_1),1]$ and $\mathcal{C}(\textbf{a})\cap[0,1-r(G_1)]=\mathcal{C}(\textbf{a})\cap I_{\underbrace{\scriptstyle00\ldots0}_{k}}$ are identical up to a shift. Moreover, $G_2$ is longer than or equal to every gap of $\mathcal{C}(\textbf{a})\cap[r(G_1),l(G_2)]\subseteq[r(G_1),1]$, and thus also $G_2$ is longer than or equal to every gap of $\mathcal{C}(\textbf{a})\cap[0,l(G_2)-r(G_1)]\subseteq[0,1-r(G_1)]$. See \cref{FSS}.
\begin{figure}
\begin{tikzpicture}[scale=10]
\draw[](0,0)--(1,0)node[pos=0,anchor=east]{$I=[0,1]=$}
node[pos=5/10,anchor=south]{$G_1$}
node[pos=9.1/10,anchor=south]{$G_2$};
\node[anchor=south west]at(9.8/10,0.03){$\mathcal{C}(\textbf{a})$};
\draw[line width=4pt,dashed](0/10,0)--(0.6/10,0);
\draw[line width=4pt,dashed](1.2/10,0)--(1.8/10,0);
\draw[line width=4pt,dashed](2.4/10,0)--(3.0/10,0);
\draw[line width=4pt,dashed](3.6/10,0)--(4.2/10,0);
\draw[line width=4pt,dashed](5.8/10,0)--(6.4/10,0);
\draw[line width=4pt,dashed](7.0/10,0)--(7.6/10,0);
\draw[line width=4pt,dashed](8.2/10,0)--(8.8/10,0);
\draw[line width=4pt,dashed](9.4/10,0)--(10/10,0);

\node[scale=2]at(0/10,0){$\textbf{[}$};
\node[scale=2]at(3/10,0){$\textbf{]}$};
\node[scale=2]at(4.2/10,0){$\textbf{]}$};
\node[scale=2]at(5.8/10,0){$\textbf{[}$};
\node[scale=2]at(8.8/10,0){$\textbf{]}$};
\node[scale=2]at(10/10,0){$\textbf{]}$};

\node[anchor=north]at(0/10,-0.03){$0$};
\node[anchor=north]at(2.5/10,-0.03){$l(G_2)-r(G_1)$};
\node[anchor=north]at(4.4/10,-0.03){$l(G_1)$};
\node[anchor=north]at(6/10,-0.03){$r(G_1)$};
\node[anchor=north]at(8.8/10,-0.03){$l(G_2)$};
\node[anchor=north]at(10/10,-0.03){$1$};

\draw[decorate,decoration=brace,thick](0/10,0+0.05)--(3/10,0+0.05)node[pos=1/2,anchor=south]{$l(G_2)-r(G_1)$};
\draw[decorate,decoration=brace,thick](5.8/10,0+0.05)--(8.8/10,0+0.05)node[pos=1/2,anchor=south]{$l(G_2)-r(G_1)$};
\end{tikzpicture}
\caption{Illustration of using the self-similarity of $\mathcal{C}(\textbf{a})$ in the proof of \cref{TAMC}. Since $[0,l(G_1)]$ and $[r(G_1),1]$ are identical up to a shift, their subintervals $[0,l(G_2)-r(G_1)]$ and $[r(G_1),l(G_2)]$ are also identical up to a shift.}\label{FSS}
\end{figure}
It again follows that
\begin{equation*}
(l(G_2)-(l(G_2)-r(G_1)),r(G_2)-0)=(r(G_1),r(G_2))
\end{equation*}
and by \cref{LLG} (ii), $(r(G_1),r(G_2)) \setminus F_2\subseteq\mathcal{C}(\textbf{a})^c-\mathcal{C}(\textbf{a})$ for some finite set $F_2$.

Generally, assume that we have defined the rightmost longest gaps $G_1,G_2,\dots,G_n$ for some $n\in\mathbb{N}$ with strictly decreasing length and such that $G_{i+1}$ lies on the right of $G_{i}$, and we have proved that the set $S\cap[0,r(G_n)]$ is finite. Let $G_{n+1}$ be the rightmost longest gap of $\mathcal{C}(\textbf{a})\cap[r(G_n),1]$. Then $G_{n+1}$ is longer than or equal to every gap of $\mathcal{C}(\textbf{a})\cap[r(G_n),l(G_{n+1})]$, and using \cref{Lbrick} similarly as earlier, we get also that $G_{n+1}$ is longer than or equal to every gap of $\mathcal{C}(\textbf{a})\cap[0,l(G_{n+1})-r(G_n)]$. Then
\begin{equation*}
(l(G_{n+1})-(l(G_{n+1})-r(G_n)),r(G_{n+1})-0)=(r(G_n),r(G_{n+1}))
\end{equation*}
and by \cref{LLG} (ii), $(r(G_n),r(G_{n+1})) \setminus F_{n+1}\subseteq\mathcal{C}(\textbf{a})^c-\mathcal{C}(\textbf{a})$ for some finite set $F_{n+1}$. This means that $S\cap[r(G_n),r(G_{n+1})]$ contains at most $\{r(G_n),r(G_{n+1})\}\cup F_{n+1}$, which is a finite set that keeps $S\cap[0,r(G_{n+1})]$ still finite. Since $\lim_{n\to\infty}r(G_n)=1$, we conclude inductively that $S\cap[0,1]$ is at most countably infinite.
\end{proof}

Concluding \cref{TALC,TAMC}, we state our main result in this section.

\begin{corollary}\label{CCI}
For every $\textbf{\em a}\in(0,1)^{\mathbb{N}}$, the set $S\coloneqq[-1,1]\setminus(\mathcal{C}(\textbf{\em a})^c-\mathcal{C}(\textbf{\em a}))$ is countably infinite.
\end{corollary}

Working on a central Cantor set $\mathcal{C}(\textbf{a})\subseteq[0,1]$, our arguments on the size of $[-1,1]\setminus(\mathcal{C}(\textbf{a})^c-\mathcal{C}(\textbf{a}))$ heavily rely on the nature of self-similarity of $\mathcal{C}(\textbf{a})$. This means that we can obtain interesting examples by slightly perturbing the self-similarity.

\section{How small can the set $[-1,1]\setminus(\mathcal{C}^c-\mathcal{C})$ be?}\label{S3}
Let $\mathcal{C}\subseteq[0,1]$ be a Cantor set. It is easy to see that $\mathcal{C}^c-\mathcal{C}$ is always as ``big'' as an open dense subset of $[-1,1]$, leaving the set $[-1,1]\setminus(\mathcal{C}^c-\mathcal{C})$ closed and nowhere dense. In the case where $\mathcal{C}$ is a central Cantor set, we have already shown that $\mathcal{C}^c-\mathcal{C}$ covers all $[-1,1]$ except for a countably infinite set. This raises a natural question:
\begin{center}
{\em Is there a Cantor set $\mathcal{C}\subseteq[0,1]$ such that $[-1,1]\setminus(\mathcal{C}^c-\mathcal{C})=\{-1,0,1\}$?}
\end{center}
In this section, we answer this question in the affirmative by constructing a Cantor $\mathcal{C}\subseteq[0,1]$ whose gaps are placed strategically. Here is the construction of such an example.

Let $c_1\in (0,1)$. We remove from the middle of the interval $[0,1]$ an open interval $G$ with length $c_1$. Denote by $I_0$ and $I_1$ the left and the right component of $[0,1]\setminus G$ respectively. Generally, we will always denote by $I_{s0}$ and $I_{s1}$ the left and the right component which will remain from the interval $I_s$ after removal of some gap $G_s$.

Let $c_2\in(0,c_1)$ be such that $c_2<\frac{1}{2}|I_0|=\frac{1}{2}|I_1|$, where $|I|$ denotes the length of the interval $I$. We remove from $I_0$ and $I_1$ open intervals $G_0$ and $G_1$, respectively, of length $c_2$ in such a way that $l(G_0)=c(I_0)$ and $r(G_1)=c(I_1)$, where $l(I)$, $r(I)$, $c(I)$ denotes the left, the right, the center point of $I$ respectively. In the next iteration, we choose a $c_3\in(0,c_2)$ such that $c_3<\frac{1}{2}|I_{00}|=\frac{1}{2}I_{11}$ and remove the open intervals $G_{00}$ of length $c_3$, $G_{01}$ of length at most $c_3$, $G_{10}$ of length at most $c_3$, $G_{11}$ of length $c_3$ from $I_{00}$, $I_{01}$, $I_{10}$, $I_{11}$, respectively, such that $l(G_{00})=c(I_{00})$, $c(G_{01})=c(I_{01})$, $c(G_{00})=c(I_{10})$, $r(G_{11})=c(I_{1})$. See \cref{FCPC}.
\begin{figure}
\begin{tikzpicture}[scale=10]
\draw[](0,0)--(1,0)node[pos=0,anchor=east]{$I=[0,1]=$};
\draw[line width=4pt](0/10,0)--(2/10,0);
\draw[line width=4pt](3.5/10,0)--(4/10,0);
\draw[line width=4pt](6/10,0)--(6.5/10,0);
\draw[line width=4pt](8/10,0)--(10/10,0);

\node[anchor=south]at(2/10,0.035){$I_{0}$};
\node[anchor=south]at(8/10,0.035){$I_{1}$};
\node[anchor=south]at(2.75/10,0){$G_0$};
\node[anchor=south]at(5/10,0){$G$};
\node[anchor=south]at(7.25/10,0){$G_1$};

\node[scale=2]at(0/10,0){$\textbf{[}$};
\node[scale=2]at(4/10,0){$\textbf{]}$};
\node[scale=2]at(6/10,0){$\textbf{[}$};
\node[scale=2]at(10/10,0){$\textbf{]}$};

\draw[decorate,decoration=brace,thick](3.5/10-0.01,0-0.02)--(2/10+0.01,0-0.02)node[pos=1/2,anchor=north]{$\displaystyle c_{2}$};
\draw[decorate,decoration=brace,thick](6/10-0.01,0-0.02)--(4/10+0.01,0-0.02)node[pos=1/2,anchor=north]{$\displaystyle c_{1}$};
\draw[decorate,decoration=brace,thick](8/10-0.01,0-0.02)--(6.5/10+0.01,0-0.02)node[pos=1/2,anchor=north]{$\displaystyle c_{2}$};

\draw[dashed,<-,line width=1pt](2/10,0-0.01)--(2/10,0-0.06)node[pos=1,anchor=north]{$l(G_{0})=c(I_{0})=\frac{1}{2}|I_{0}|$};
\draw[dashed,<-](8/10,0-0.01)--(8/10,0-0.06)node[pos=1,anchor=north]{$r(G_1)=c(I_1)$};

\draw[](0,0-0.25)--(1,0-0.25)node[pos=0,anchor=east]{$I_0=$};
\draw[line width=4pt](0*2.5/10,0-0.25)--(2.5/10,0-0.25);
\draw[line width=4pt](4/10,0-0.25)--(5/10,0-0.25);
\draw[line width=4pt](8.75/10,0-0.25)--(9/10,0-0.25);
\draw[line width=4pt](9.75/10,0-0.25)--(10/10,0-0.25);
\node[anchor=south]at(2.5/10,0.035-0.25){$I_{00}$};
\node[anchor=south]at(9.375/10,0.035-0.25){$I_{01}$};

\node[anchor=south]at(3.25/10,0-0.25){$G_{00}$};
\node[anchor=south]at(2.75/10*2.5,0-0.25){$G_0$};
\node[anchor=south]at(9.375/10,0-0.25){$G_{01}$};

\node[scale=1]at(0/10,0-0.25){$\textbf{[}$};
\node[scale=1]at(5/10,0-0.25){$\textbf{]}$};
\node[scale=1]at(8.75/10,0-0.25){$\textbf{[}$};
\node[scale=1]at(10/10,0-0.25){$\textbf{]}$};

\draw[dashed,<-](2.5/10,0-0.01-0.25)--(2.5/10,0-0.06-0.25)node[pos=1,anchor=north]{$l(G_{00})=c(I_{00})=\frac{1}{2}|I_{00}|$};
\draw[decorate,decoration=brace,thick](4/10-0.01,0-0.02-0.25)--(2.5/10+0.01,0-0.02-0.25)node[pos=1/2,anchor=north]{$\displaystyle c_{3}$};
\draw[decorate,decoration=brace,thick](8.75/10-0.01,0-0.02-0.25)--(5/10+0.01,0-0.02-0.25)node[pos=1/2,anchor=north]{$\displaystyle c_{2}$};

\draw[decorate,decoration=brace,thick](9.75/10-0.01,0-0.02-0.25)--(9/10+0.01,0-0.02-0.25)node[pos=1/2,anchor=north]{$|G_{01}|\leq c_3$};
\node[anchor=north]at(9/10,-0.07-0.25){$c(G_{01})=c(I_{01})$};

\draw[dashed,->,line width=1pt](0,0-0.02)--(0,-0.25+0.02);
\draw [dashed,->,line width=1pt](2/10,-0.1-0.02)to[out=-90,in=175](2.5/10,-0.15)to[out=0,in=180](4/10,-0.16)to[out=-5,in=90](5/10,-0.25+0.02);
\draw [dashed,->,line width=1pt](3.5/10,0-0.02)to[out=-85,in=175](5.8/10,-0.1-0.05)to[out=-5,in=175](7.7/10,-0.1-0.06)to[out=-5,in=90](8.75/10,-0.25+0.02);
\draw [dashed,->,line width=1pt](4/10,0-0.02)to[out=-85,in=175](6/10,-0.1-0.02)to[out=-5,in=175](8/10,-0.1-0.03)to[out=-5,in=90](1,-0.25+0.02);
\end{tikzpicture}
\caption{The construction of $\mathcal{C}\subseteq[0,1]$ starts with removing $G=(\frac{1-c_1}{2},\frac{1+c_1}{2})$ at the center of $I$. The two remaining intervals are denoted by $I_0$ on the left and $I_1$ on the right. In the next step, $l(G_{00})$, $c(I_{00})$ are aligned within $I_{00}$, and $I_{01}$, $c(G_{01})$ are aligned within $I_{01}$.}\label{FCPC}
\end{figure}

Assume that for some $n\in\mathbb{N}$, we have defined intervals $I_{s_1s_2\ldots s_n}$, where $s_1s_2\ldots s_n$ is a binary sequence of length $n$, along with a decreasing sequence of positives numbers $(c_i)_{i=1}^n$. Let $c_{n+1}\in (0,c_n)$ be such that
\begin{equation*}
c_{n+1}<\frac{1}{2}|I_{\underbrace{{\scriptstyle00\ldots0}}_{n}}|=\frac{1}{2}|I_{\underbrace{{\scriptstyle11\ldots1}}_{n}}|\mbox{.}
\end{equation*}
We remove from $I_{\underbrace{{\scriptstyle00\ldots0}}_{n}}$ and $I_{\underbrace{{\scriptstyle11\ldots1}}_{n}}$ open intervals $G_{\underbrace{{\scriptstyle00\ldots0}}_{n}}$ and $G_{\underbrace{{\scriptstyle11\ldots1}}_{n}}$, respectively, each of length $c_{n+1}$, in such a way that
\begin{center}
$l(G_{\underbrace{{\scriptstyle00\ldots0}}_{n}})=c(I_{\underbrace{{\scriptstyle00\ldots0}}_{n}})$ and $r(G_{\underbrace{{\scriptstyle11\ldots1}}_{n}})=c(I_{\underbrace{{\scriptstyle11\ldots1}}_{n}})$.
\end{center}
From the remaining intervals $I_{s_1s_2\ldots s_n}$, where the binary sequence $s_1s_2\ldots s_n$ is neither all zeros nor all ones, we also remove some open intervals $G_{s_1s_2\ldots s_n}$ of length at most $c_{n+1}$. Each such gap is concentric within its respective interval, that is $c(G_{s_1s_2\ldots s_n})=c(I_{s_1s_2\ldots s_n})$. Let
\begin{equation*}
\mathscr{C}\coloneqq\bigcap_{n\in\mathbb{N}}\bigcup_{\textbf{s}\in\{0,1\}^n}I_\textbf{s}\mbox{.}
\end{equation*}
We claim that $\mathscr{C}\subseteq[0,1]$ is a Cantor set. Indeed, it is clearly a perfect set containing both $0$ and $1$. Moreover, since all the gaps are placed near the centers of intervals, the lengths $|I_{s_1s_2\ldots s_n}|$ shrink geometrically to zero as $n\to\infty$. In particular, they follow the recursive inequality $\max\{|I_{s_1s_2\ldots s_n0}|,|I_{s_1s_2\ldots s_n1}|\}\leq\frac{1}{2}|I_{s_1s_2\ldots s_n}|$. Therefore, $\mathscr{C}$ is nowhere dense and hence qualifies as a Cantor set.

Before going into the next theorem, we would like to highlight that three key properties of the Cantor set $\mathscr{C}$. They are the founding stones of the next theorem.
\begin{itemize}
\item[•] $G_{\underbrace{{\scriptstyle00\ldots0}}_{n}}$ is always strictly longer than every gap of $\mathscr{C}\cap[0,l(G_{\underbrace{{\scriptstyle00\ldots0}}_{n}})]$.
\item[•] $G_{\underbrace{{\scriptstyle00\ldots0}}_{n}}$ and $G_{\underbrace{{\scriptstyle11\ldots1}}_{n}}$ always have the same length.
\item[•] $I_{\underbrace{{\scriptstyle00\ldots0}}_{n}}$ and $I_{\underbrace{{\scriptstyle11\ldots1}}_{n}}$ always have the same length.
\end{itemize}

\begin{theorem}\label{TS3}
There is a Cantor set $\mathcal{C}\subseteq[0,1]$ such that
\begin{equation*}
[-1,1]\setminus(\mathcal{C}^c-\mathcal{C})=\{0,\pm1\}\mbox{.}
\end{equation*}
\end{theorem}
\begin{proof}
Let $\mathcal{C}\coloneqq\mathscr{C}$ constructed above. We will show that $(0,1)\subseteq\mathcal{C}^c-\mathcal{C}$. The argument for $(-1,0)\subseteq\mathcal{C}^c-\mathcal{C}$ follows symmetrically.

Since $G$ is strictly longer than every gap of $\mathcal{C}\cap[0,l(G)]$, we have, by (i) of \cref{LLG}, that
\begin{equation*}
(l(G)-l(G),r(G)-0)=(0,r(G))\subseteq\mathcal{C}^c-\mathcal{C}\mbox{.}
\end{equation*}
Similarly, since every gap on the left of $G_0$, that is, gap of $\mathcal{C}\cap[0,l(G_0)]$, is strictly shorter than $G_0$ and thus $G_1=(l(G_1),r(G_1))$, we have, by (i) of \cref{LLG}, that
\begin{equation*}
(l(G_1)-l(G_0),r(G_1)-0)=(l(G_1)-l(G_0),r(G_1))\subseteq\mathcal{C}^c-\mathcal{C}\mbox{.}
\end{equation*}
Also, note that
\begin{equation*}
l(G_0)=\frac{1}{2}|I_0|=\frac{1}{2}|I_1|=c(I_1)-l(I_1)=r(G_1)-r(G)\mbox{,}
\end{equation*}
and that $l(G_1)-r(G_1)<0$. See \cref{FII}.
\begin{figure}
\begin{tikzpicture}[scale=10]
\draw[](0,0)--(1,0)node[pos=0,anchor=east]{$I=[0,1]=$}
node[pos=1,anchor=south west]{$\mathscr{C}$};
\draw[line width=4pt,dashed](0/10,0)--(2/10,0);
\draw[line width=4pt,dashed](3.5/10,0)--(4/10,0);
\draw[line width=4pt,dashed](6/10,0)--(6.5/10,0);
\draw[line width=4pt,dashed](8/10,0)--(10/10,0);
\node[anchor=south]at(2/10,0.035){$I_{0}$};
\node[anchor=south]at(8/10,0.035){$I_{1}$};
\node[anchor=south]at(2.75/10,0){$G_0$};
\node[anchor=south]at(5/10,0){$G$};
\node[anchor=south]at(7.25/10,0){$G_1$};

\node[scale=2]at(0/10,0){$\textbf{[}$};
\node[scale=2]at(4/10,0){$\textbf{]}$};
\node[scale=2]at(6/10,0){$\textbf{[}$};
\node[scale=2]at(10/10,0){$\textbf{]}$};

\draw[<-,line width=1pt](2/10,0-0.01)--(2/10,0-0.04)node[pos=1,anchor=north]{$l(G_0)=\frac{1}{2}|I_0|=\frac{1}{2}|I_1|$};
\draw[<-,line width=1pt](6/10,0-0.04)--(5.9/10,0-0.07);
\node[anchor=north]at(5.8/10,0-0.07){$l(I_1)=r(G)$};
\draw[<-,line width=1pt](6.5/10,0+0.01)--(6.5/10,0+0.04)node[pos=1,anchor=south]{$l(G_1)$};
\draw[<-,line width=1pt](8/10,0-0.01)--(8.2/10,0-0.07);
\node[anchor=north]at(8.4/10,0-0.07){$c(I_1)=r(G_1)$};

\draw[decorate,decoration=brace,thick](8/10-0.01,0-0.02)--(6/10+0.01,0-0.02)node[pos=1/2,anchor=north]{$\frac{1}{2}|I_1|$};
\end{tikzpicture}
\caption{Illustration of showing $l(G_0)=r(G_1)-r(G)$ and $l(G_1)-r(G_1)<0$ in the proof of \cref{TS3}.}\label{FII}
\end{figure}
The inequality
\begin{equation*}
l(G_1)-l(G_0)=l(G_1)-(r(G_1)-r(G))=l(G_1)-r(G_1)+r(G)<r(G)
\end{equation*}
shows that the right endpoint of $(0,r(G))$ is strictly greater than the left endpoint of $(l(G_1)-l(G_0),r(G_1))$. It follows that $(0,r(G))\cup(l(G_1)-l(G_0),r(G_1))=(0,r(G_1))$.

Inductively, we can show that for any $n\in\mathbb{N}$, the interval $(0,r(G_{\underbrace{\scriptstyle11\ldots1}_{n}}))\subseteq\mathcal{C}^c-\mathcal{C}$. Moreover, since $r(G_{\underbrace{\scriptstyle11\ldots1}_{n}})\to1$ as $n\to\infty$, it follows that $(0,1)\subseteq\mathcal{C}^c-\mathcal{C}$.
\end{proof}

Using this particular Cantor set $\mathscr{C}\subseteq[0,1]$, $\mathcal{C}^c-\mathcal{C}$ is maximized, covering all of $[-1,1]\setminus\{-1,0,1\}$. In the next section, we shift focus in the opposite direction and explore how to minimize $\mathcal{C}^c-\mathcal{C}$ in sense of Lebesgue measure.

\section{Measure of $[-1,1]\setminus(\mathcal{C}^c-\mathcal{C})$}\label{S4}
Recall that the classical Cantor ternary $\mathfrak{C}\subseteq[0,1]$ is ``small'' in both the sense of Baire category and Lebesgue measure, that is, it is meager and has measure zero. Consequently, its complement $\mathfrak{C}^c\subseteq[0,1]$ is ``big'' in both senses, that is, it is comeager and has full measure in $[0,1]$. It follows that $\mathfrak{C}^c-\mathfrak{C}=\bigcup_{t\in\mathfrak{C}}\mathfrak{C}^c-t$ must also be ``big'' in both senses in $[-1,1]$. However, unlike the classical Cantor set, a Cantor $\mathcal{C}\subseteq[0,1]$ in general may have positive Lebesgue measure. This leads to the following natural question:
\begin{center}
{\em Given $\mathcal{C}$ of varying ``fatness'', is $\mathcal{C}^c-\mathcal{C}$ necessarily of full measure in $[-1,1]$?}
\end{center}
From the perspective of Lebesgue measure, it is particularly interesting that our findings suggest a stark contrast in the behavior of $\mathcal{C}^c-\mathcal{C}$ in $[-\frac{1}{2},\frac{1}{2}]$ and in $[-1,-\frac{1}{2}]\cup[\frac{1}{2},1]$. By $m(A)$, we will denote the Lebesgue measure of a set $A\subseteq\mathbb{R}$. Note that the Steinhaus theorem plays a crucial role in this section,\footnote{In its classical form, it states that if $m(A)>0$, then $A-A$ contains an open neighbourhood of zero. See \cite[Th\'{e}or\`{e}me VIII]{Steinhaus1920} or \cite{MR4080557}.} as a variant of it is used in the proofs of \cref{TS1212,TS1234,RB}. In particular, the version we use states that if $m(A)>0$ and $m(B)>0$, then $A+B$ contains an interval.\footnote{In fact, Steinhaus stated in \cite[Th\'{e}or\`{e}me VII]{Steinhaus1920} that if $m(A)>0$ and $m(B)>0$, then $A-B$ contains an interval. The corresponding result for $A+B$ appears as an exercise in \cite[Exercise 5, Chapter 7]{MR0924157}.}

\begin{theorem}\label{TS1212}
Let $\mathcal{C}\subseteq[0,1]$ be a Cantor set, and define $S\coloneqq[-1,1]\setminus(\mathcal{C}^c-\mathcal{C})$. Then $S\cap[-\frac{1}{2},\frac{1}{2}]$ has Lebesgue measure zero.
\end{theorem}  
\begin{proof}
Suppose that $S\cap[0,\frac{1}{2}]$ has positive Lebesgue measure. Since $S\subseteq\mathcal{C}\cup(\mathcal{C}-1)$, it follows that $\mathcal{C}\cap[0,\frac{1}{2}]$ also has positive Lebesgue measure. Let $Y=S\cap[0,\frac{1}{2}]$. we have, by \cref{LS}, that
\begin{equation*}
\textstyle(\mathcal{C}\cap[0,\frac{1}{2}])+Y=((\mathcal{C}\cap[0,\frac{1}{2}])+Y)\cap[0,1]\subseteq(\mathcal{C}+Y)\cap[0,1]\subseteq\mathcal{C}\mbox{,}
\end{equation*}
so $(\mathcal{C}\cap[0,\frac{1}{2}])+Y\subseteq\mathcal{C}$. Since $\mathcal{C}\cap[0,\frac{1}{2}]$ and $Y$ both have positive Lebesgue measure, their sum contains an interval by Steinhaus theorem. Hence, $\mathcal{C}$ contains an interval, which leads to a contradiction.

Similarly, having positive measure in $S\cap[-\frac{1}{2},0]$ also leads to a contradiction.
\end{proof}

As described in \cref{TS1212}, the set $S\cap[-\frac{1}{2},\frac{1}{2}]$ is always ``small'' in the sense of Lebesgue measure. In particular,
\begin{center}
{\em $\mathcal{C}^c-\mathcal{C}$ always has full Lebesgue measure in $[-\frac{1}{2},\frac{1}{2}]$ regardless of the ``fatness'' of $\mathcal{C}$.}
\end{center}

While $S$ may have positive Lebesgue measure outside this central interval $[-\frac{1}{2},\frac{1}{2}]$, some symmetry constraints still apply. We will first address these consideration in \cref{TS1234}. Finally, in \cref{CSPM}, we will show that the Lebesgue measure of $S$ can be as big as $\frac{1}{2}$. In other words,
\begin{center}
{\em $\mathcal{C}^c-\mathcal{C}$ does not necessarily have full Lebesgue measure in $[-1,-\frac{1}{2}]\cup[\frac{1}{2},1]$.}
\end{center}

\begin{theorem}\label{TS1234}
Let $\mathcal{C}\subseteq[0,1]$ be a Cantor set, and define $S\coloneqq[-1,1]\setminus(\mathcal{C}^c-\mathcal{C})$.
\begin{enumerate}[(i)]
\item $S\cap[-1,-\frac{3}{4}]$ and $S\cap[\frac{1}{2},\frac{3}{4}]$ cannot both have positive Lebesgue measure.
\item $S\cap[-\frac{3}{4},-\frac{1}{2}]$ and $S\cap[\frac{3}{4},1]$ cannot both have positive Lebesgue measure.
\end{enumerate}
\end{theorem}
\begin{proof}
Suppose both $S\cap[-1,-\frac{3}{4}]$ and $S\cap[\frac{1}{2},\frac{3}{4}]$ have positive Lebesgue measure. Since $S\subseteq\mathcal{C}\cup(\mathcal{C}-1)$, it follows that $\mathcal{C}\cap[0,\frac{1}{4}]$ also has positive Lebesgue measure. See \cref{FS1234}.
\begin{figure}
\begin{tikzpicture}[scale=10]
\draw[](0,0)--(1,0)node[pos=0,anchor=east]{$[-1,1]=$};
\draw[->,line width=4pt,dotted](0/10,0)--(1.25/10,0);
\draw[->,line width=4pt,dotted](5/10,0)--(6.25/10,0);
\draw[line width=4pt,dashed](7.5/10,0)--(8.75/10,0);
\node[anchor=south]at(0.625/10,0.07){$S\cap[-1,-\frac{3}{4}]\subseteq\mathcal{C}-1$};
\node[anchor=south]at(5.625/10,0.07){$(S\cap[-1,-\frac{3}{4}])+1\subseteq\mathcal{C}$};
\node[anchor=south]at(8.125/10,0.12){$Y\coloneqq S\cap[\frac{1}{2},\frac{3}{4}]$};
\draw[decorate,decoration=brace,thick](10/10-0.01,0-0.03)--(5/10+0.01,0-0.03)node[pos=1/2,anchor=north]{$\mathcal{C}$};
\draw[decorate,decoration=brace,thick](5/10-0.01,0-0.03)--(0/10+0.01,0-0.03)node[pos=1/2,anchor=north]{$\mathcal{C}-1$};
\draw[line width=2pt](0/10,0.01)--(0/10,-0.01)node[pos=0,anchor=south]{$-1$};
\draw[](1.25/10,0.01)--(1.25/10,-0.01)node[pos=0,anchor=south]{$-\frac{3}{4}$};
\draw[line width=2pt](2.5/10,0.01)--(2.5/10,-0.01)node[pos=0,anchor=south]{$-\frac{1}{2}$};
\draw[](3.75/10,0.01)--(3.75/10,-0.01)node[pos=0,anchor=south]{$-\frac{1}{4}$};
\draw[line width=2pt](5/10,0.01)--(5/10,-0.01)node[pos=0,anchor=south]{$0$};
\draw[](6.25/10,0.01)--(6.25/10,-0.01)node[pos=0,anchor=south]{$\frac{1}{4}$};
\draw[line width=2pt](7.5/10,0.01)--(7.5/10,-0.01)node[pos=0,anchor=south]{$\frac{1}{2}$};
\draw[](8.75/10,0.01)--(8.75/10,-0.01)node[pos=0,anchor=south]{$\frac{3}{4}$};
\draw[line width=2pt](10/10,0.01)--(10/10,-0.01)node[pos=0,anchor=south]{$1$};
\end{tikzpicture}
\caption{Illustration of the setup of the proof of \cref{TS1234}.}\label{FS1234}
\end{figure}
Let $Y\coloneqq S\cap[\frac{1}{2},\frac{3}{4}]$. we have, by \cref{LS}, that
\begin{equation*}
\textstyle(\mathcal{C}\cap[0,\frac{1}{4}])+Y=((\mathcal{C}\cap[0,\frac{1}{4}])+Y)\cap[0,1]\subseteq(\mathcal{C}+Y)\cap[0,1]\subseteq\mathcal{C}\mbox{,}
\end{equation*}
so $(\mathcal{C}\cap[0,\frac{1}{4}])+Y\subseteq\mathcal{C}$. Since $\mathcal{C}\cap[0,\frac{1}{4}]$ and $Y$ both have positive Lebesgue measure, their sum contains an interval by Steinhaus theorem. Hence, $\mathcal{C}$ contains an interval, which leads to a contradiction.

Condition (ii) can be proved in the same way.
\end{proof}

\begin{corollary}\label{CS1234}
Let $\mathcal{C}\subseteq[0,1]$ be a Cantor set, and define $S\coloneqq[-1,1]\setminus(\mathcal{C}^c-\mathcal{C})$.
\begin{enumerate}[(i)]
\item If $m(S\cap[\frac{1}{2},1])>\frac{1}{4}$, then $m(S\cap[-1,-\frac{1}{2}])=0$.
\item If $m(S\cap[-1,-\frac{1}{2}])>\frac{1}{4}$, then $m(S\cap[\frac{1}{2},1])=0$.
\end{enumerate}
\end{corollary}
\begin{proof}
To see (i), suppose $m(S\cap[-1,-\frac{1}{2}])>0$. Then, at least one of the sets $S\cap[-1,-\frac{3}{4}]$ or $S\cap[-\frac{3}{4},-\frac{1}{2}]$ must have positive Lebesgue measure. By \cref{TS1234}, it follows that at least one of $S\cap[\frac{1}{2},\frac{3}{4}]$ or $S\cap[\frac{3}{4},1]$ must have zero Lebesgue measure. Consequently, the Lebesgue measure of $S\cap[\frac{1}{2},1]$ is not greater than $\frac{1}{4}$, contradicting the assumption.

The arguments for (ii) follow by identical reasoning.
\end{proof}

\begin{corollary}\label{CSMB}
Let $\mathcal{C}\subseteq[0,1]$ be a Cantor set, and define $S\coloneqq[-1,1]\setminus(\mathcal{C}^c-\mathcal{C})$. Then
\begin{center}
$0\leq m(S)<\frac{1}{2}$, or equivalently, $\frac{3}{2}<m(\mathcal{C}^c-\mathcal{C})\leq2$.
\end{center}
\end{corollary}
\begin{proof}
To show that $m(S)$ can be as small as zero, take the classical Cantor ternary set or any example mentioned in \cref{S2,S3}. 

On the other hand, we now decide an upper bound for $m(S)$. By \cref{TS1212}, we have $m(S)=m(S\cap[-1,-\frac{1}{2}])+m(S\cap[\frac{1}{2},1])\leq1$. In addition, incorporating \cref{TS1234,CS1234} on
\begin{align*}
\textstyle m(S)=m(S\cap[-1,-\frac{3}{4}])+m(S\cap[-\frac{3}{4},-\frac{1}{2}])+m(S\cap[\frac{1}{2},\frac{3}{4}])+m(S\cap[\frac{3}{4},1])\mbox{,}
\end{align*}
it is easy to see that $m(S)\leq\frac{1}{2}$ case by case.

Lastly, we rule out the case where $m(S)=\frac{1}{2}$. Suppose $m(S)=\frac{1}{2}$, and again incorporate \cref{TS1234,CS1234}. It is easy to see that two out of the four sets $S\cap[-1,-\frac{3}{4}]$, $S\cap[-\frac{3}{4},-\frac{1}{2}]$, $S\cap[\frac{1}{2},\frac{3}{4}]$, and $S\cap[\frac{3}{4},1]$ must have zero Lebesgue measure, forcing the other two to have full Lebesgue measure. However, $S$ cannot have full Lebesgue measure in any nontrivial subinterval in $[-1,1]$, because, by definition, $S$ is the complement of a dense open set $\bigcup_{t\in\mathcal{C}}\mathcal{C}^c-t$, which has positive Lebesgue measure in every nontrivial subinterval in $[-1,1]$.
\end{proof}

So, we know that $\frac{1}{2}$ is an upper bound for the Lebesgue measure of $S$. But we still do not clearly know whether $S$ can have positive Lebesgue measure or not. In what follows, we will go through two theorems that describe a way to increase the ``size'' of $S$, and ultimately show that $\frac{1}{2}$ is the least upper bound for $m(S)$ in \cref{CSPM}.

\begin{theorem}\label{TAB}
Let $A\subseteq[0,\frac{1}{2}]$ be a Cantor set. If $B\subseteq[0,\frac{1}{2}]$ is a set such that $A+B\subseteq[0,1]$ is also a Cantor set, then there exists a Cantor set $\mathcal{C}\subseteq[0,1]$ such that $B+\frac{1}{2}\subseteq S\cap[\frac{1}{2},1]$, where $S\coloneqq[-1,1]\setminus(\mathcal{C}^c-\mathcal{C})$.
\end{theorem}
\begin{proof}
Let
\begin{center}
$\mathcal{C}\coloneqq A\cup E$, where $E\coloneqq(A+B+\frac{1}{2})\cap[\frac{1}{2},1]$.
\end{center}
It is easy to see that $\mathcal{C}$ is a Cantor set of $[0,1]$ due to its construction. See \cref{FAB}.
\begin{figure}[h]
\begin{tikzpicture}[scale=10]
\draw[](0,0)--(1,0)
node[pos=0,anchor=east]{$[0,\frac{3}{2}]=$};
\draw[line width=4pt,dashed](0,0)--(1/3,0);
\draw[line width=3pt,densely dotted](1/3,0)--(1,0);
\node[anchor=south]at(1/6,0.03){$A$ is Cantor};
\node[anchor=south]at(2/3,0.02){$A+B+\frac{1}{2}$ is Cantor};
\draw[decorate,decoration=brace,thick](1/3-0.01,0-0.02)--(0+0.01,0-0.02)node[pos=1/2,anchor=north]{$\displaystyle A$};
\draw[decorate,decoration=brace,thick](2/3-0.01,0-0.02)--(1/3+0.01,0-0.02)node[pos=1/2,anchor=north]{$\displaystyle E$};
\draw[line width=1pt,dashed](0/3,0-0.02)--(0/3,0+0.06)node[pos=0,anchor=north]{$0$};
\draw[line width=1pt,dashed](1/3,0-0.02)--(1/3,0+0.06)node[pos=0,anchor=north]{$\frac{1}{2}$};
\draw[line width=1pt,dashed](2/3,0-0.02)--(2/3,0+0.02)node[pos=0,anchor=north]{$1$};
\draw[line width=1pt,dashed](3/3,0-0.02)--(3/3,0+0.06)node[pos=0,anchor=north]{$\frac{3}{2}$};
\draw[decorate,decoration=brace,thick](2/3-0.01,0-0.09)--(0+0.01,0-0.09)node[pos=1/2,anchor=north]{$\displaystyle\mathcal{C}\coloneq A\cup E$};
\end{tikzpicture}
\caption{Illustration of the construction of the Cantor set $\mathcal{C}\subseteq[0,1]$ described in \cref{TAB}.}\label{FAB}
\end{figure}

We will show that $B+\frac{1}{2}\subseteq S\cap[\frac{1}{2},1]$. Indeed,
\begin{align*}
\textstyle(\mathcal{C}+B+\frac{1}{2})\cap[0,1]=((A\cup E)+B+\frac{1}{2})\cap[\frac{1}{2},1]&\textstyle=(A+B+\frac{1}{2})\cap[\frac{1}{2},1]\\
&=E\subseteq\mathcal{C}\mbox{.}
\end{align*}
Since $(C+B+\frac{1}{2})\cap[0,1]\subseteq\mathcal{C}$, we have $B+\frac{1}{2}\subseteq S$ by \cref{LS}. Also, since $B+\frac{1}{2}\subseteq[\frac{1}{2},1]$, we further conclude that $B+\frac{1}{2}\subseteq S\cap[\frac{1}{2},1]$.
\end{proof}
Now, we can use \cref{TAB} to show that the set $[-1,1]\setminus(\mathcal{C}^c-\mathcal{C})$ can contain Cantor sets of various types. Actually, we have even more general result.

\begin{theorem}\label{TB}
For every compact meager set $B\subseteq[0,\frac{1}{2}]$ containing $0$ and $\frac{1}{2}$, there exists a Cantor set $\mathcal{C}\subseteq[0,1]$ such that $B+\frac{1}{2}\subseteq S\cap[\frac{1}{2},1]$, where $S\coloneqq[-1,1]\setminus(\mathcal{C}^c-\mathcal{C})$.
\end{theorem}
\begin{proof}
Since $B\cup(B+\frac{1}{2})$ does not contain any interval, we can choose a countable dense set $D$ in $\mathbb{R}\setminus(B\cup(B+\frac{1}{2}))$. Note that $0,\frac{1}{2}\not\in D-B$. Otherwise, $D\cap B\neq\emptyset$ or $D\cap(B+\frac{1}{2})\neq\emptyset$, which contradicts that $D\subseteq\mathbb{R}\setminus(B\cup(B+\frac{1}{2}))$. The set $D-B=\bigcup_{d\in D}d-B$ is a countable union of meager sets, and so it is also meager. Hence, $\mathbb{R}\setminus(D-B)$ is comeager in $\mathbb{R}$, and therefore contains a dense $G_\delta$ subset of $\mathbb{R}$. Since it is Borel and uncountable in every nontrivial closed interval, it also contains a Cantor set $A\subseteq[0,\frac{1}{2}]$, by the perfect set theorem for Borel sets.\footnote{See \cite[Theorem 13.6]{MR1321597}. Note that we additionally require from $A$ to contain $0$ and $\frac{1}{2}$ in this paper. That is why our $D$ is chosen in such a way to ensure $0,\frac{1}{2}\not\in D-B$, and therefore $0,\frac{1}{2}\in\mathbb{R}\setminus(D-B)$.} Note that $A$ is chosen from $\mathbb{R}\setminus(D-B)$, and therefore $A\cap(D-B)=\emptyset$.

With the Cantor set $A\subseteq[0,\frac{1}{2}]$ determined, we claim that $A+B\subseteq[0,1]$ is also a Cantor set. First, it is easy to see that $A+B$ is a perfect subset of $[0,1]$ containing $0$ and $1$. In particular, $A+B$ is closed, so to show that it is also nowhere dense, it suffices to prove that $A+B$ has empty interior. On the contrary, suppose that $A+B$ contains some nontrivial interval. Then this interval has nonempty intersection with the dense set $D$. Hence $(A+B)\cap D\neq\emptyset$. Then there exist some $a\in A$, $b\in B$, $d\in D$ such that $a+b=d$. This implies that $a=d-b\in D-B$, contradicting the fact that $A\cap(D-B)=\emptyset$.

Finally, since both $A\subseteq[0,\frac{1}{2}]$ and $A+B\subseteq[0,1]$ are Cantor sets, by \cref{TAB}, there is a Cantor set $\mathcal{C}\subseteq[0,1]$ such that $B+\frac{1}{2}\subseteq S\cap[\frac{1}{2},1]$, where $S\coloneqq[-1,1]\setminus(\mathcal{C}^c-\mathcal{C})$.
\end{proof}

To ultimately show that $S$ can have positive Lebesgue measure, we begin with a Cantor set $B\subseteq[0,\frac{1}{2}]$ of positive Lebesgue measure. We want to carefully verify that the proof of \cref{TB} remains valid in this context. One might wonder: {\em what if the other Cantor set $A\subseteq[0,\frac{1}{2}]$, chosen from $\mathbb{R}\setminus(D-B)$, also has positive Lebesgue measure?} This would be catastrophic, as $A+B$ would then contain an interval by Steinhaus Theorem, and thus could not be a Cantor set. If $A+B$ fails to be a Cantor set, then \cref{TAB} would no longer apply, and the entire argument would fall apart. The following remark ensures that such a scenario cannot occur.

\begin{remark}\label{RB}
{\em If $B\subseteq[\frac{1}{2},1]$ has positive Lebesgue measure and $D$ is a dense subset of $\mathbb{R}$, then $\mathbb{R}\setminus(D-B)$ must have Lebesgue measure zero.}
\end{remark}
\begin{proof}
Suppose $\mathbb{R}\setminus(D-B)$ has positive Lebesgue measure, then $(\mathbb{R}\setminus(D-B))+B$ contains an interval, by Steinhaus theorem. Since $D$ is dense, there is $q\in D$ such that $q\in(\mathbb{R}\setminus(D-B))+B$ which leads to a contradiction that $q-b\in\mathbb{R}\setminus(D-B)$ for some $b\in B$.\footnote{The arguments are identical to those used in the proof of \cite[Proposition 7]{MR4559372}. In fact, this paper is originally motivated by our initial efforts to search for the $F_\sigma$ set described in \cite[Remark 1]{MR4559372}.}
\end{proof}

Finally, we state the main results of this section.

\begin{corollary}\label{CSFC}
There is a Cantor set $\mathcal{C}\subseteq[0,1]$ such that
\begin{center}
$[-1,1]\setminus(\mathcal{C}^c-\mathcal{C})$ contains a  Cantor set of positive Lebesgue measure.
\end{center}
\end{corollary}

\begin{corollary}\label{CSPM}
Let $\mathcal{C}\subseteq[0,1]$ be a Cantor set, and define $S\coloneqq[-1,1]\setminus(\mathcal{C}^c-\mathcal{C})$. Then
\begin{center}
$\sup(m(S))=\frac{1}{2}$, or equivalently, $\inf(m(\mathcal{C}^c-\mathcal{C}))=\frac{3}{2}$.
\end{center}
\end{corollary}
\begin{proof}
Indeed, it is well known that for every $\varepsilon\in(0,\frac{1}{2})$, there exists a Cantor set $B\subseteq[0,\frac{1}{2}]$ such that $m(B)>\frac{1}{2}-\varepsilon$. The Cantor set $B\subseteq[0,\frac{1}{2}]$ is compact, meager and contains $0$ and $\frac{1}{2}$. By \cref{TB}, there is a Cantor set $\mathcal{C}\subseteq[0,1]$ such that $B+\frac{1}{2}\subseteq S\cap[\frac{1}{2},1]$. Since
\begin{align*}
\textstyle m(S)\geq m(S\cap[\frac{1}{2},1])\geq m(B+\frac{1}{2})=m(B)>\frac{1}{2}-\varepsilon\mbox{,}
\end{align*}
we get that $\frac{1}{2}$ is the least upper bound for $m(S)$.
\end{proof}

In the end, let us revisit \cref{TB}. As discussed, the set $S\coloneqq[-1,1]\setminus(\mathcal{C}^c-\mathcal{C})$ can contain Cantor sets of various type. Recall that the set $S$ is the complement of a dense open set $\bigcup_{t\in\mathcal{C}}\mathcal{C}^c-t$ and is therefore closed and nowhere dense. This motivated our final question:
\begin{center}
{\em Can the set $S$ itself be a Cantor set?}
\end{center}
Our final theorem shows that it cannot.

\begin{theorem}\label{RNC}
Let $\mathcal{C}\subseteq[0,1]$ be a Cantor set, and define $S\coloneqq[-1,1]\setminus(\mathcal{C}^c-\mathcal{C})$. Then
\begin{center}
$S$ cannot be a Cantor set.
\end{center}
\end{theorem}
\begin{proof}
In particular, $0$ is always an isolated point of $S$. To see this, let $G$ be any gap of $\mathcal{C}$, and so $l(G),r(G)\in\mathcal{C}$. Trivially, $G$ is strictly longer than every gap of $\mathcal{C}\cap[l(G),l(G)]$ as well as every gap of $\mathcal{C}\cap[r(G),r(G)]$. It then follows from (i) of \cref{LLG} that $(-|G|,0)\cup(0,|G|)\subseteq\mathcal{C}^c-\mathcal{C}$. Therefore, $0\in S$ is isolated.
\end{proof}

\begin{remark}\label{RSD}
{\em By the Cantor-Bendixson theorem (see \cite[Theorem 6.4]{MR1321597}), every closed set can be uniquely presented as a union of two disjoint sets: a countable one and a perfect one. So, $S=A\cup B$, where $A$ is a countable set and $B$ is a perfect set. Since $S$ is nowhere dense, $B$ is also nowhere dense, and thus it is either a Cantor set or an empty set.}
\end{remark}

\begin{remark}\label{RNE}
{\em The Cantor set part in the decomposition described in \cref{RSD} may be empty as we could see, for example, in \cref{CCI,TS3}. However, the countable part cannot be empty, as it must contain all isolated points of $S$.}
\end{remark}

\bibliographystyle{amsplain}

\end{document}